\documentclass{opt2022} 

\usepackage{forloop}
\usepackage{nicefrac}
\usepackage{hyperref}
\usepackage{enumitem}
\usepackage{mathtools,cases}

\newcommand{\PAPERTITLE}{Sufficient conditions for non-asymptotic convergence of Riemannian optimisation methods}

\newcommand{\subscript}[2]{$#1 _ #2$}
\newlist{assumplist1}{enumerate}{1}
\setlist[assumplist1]{label=(\subscript{\textbf{A}}{{\arabic*}})}

\newcommand{\argmin}{\mathop{\mathrm{argmin}}}

\newcommand{\grad}{\mathrm{grad}}
\newcommand{\Hess}{\mathrm{Hess}}
\newcommand{\Exp}{\mathrm{Exp}}
\newcommand{\Log}{\mathrm{Log}}
\newcommand{\diam}{\mathrm{diam}}
\newcommand{\argmax}{\mathop{\mathrm{argmax}}}

\newcommand{\defcal} [1]{\expandafter\newcommand\csname cal#1\endcsname{{\cal #1}}}
\newcommand{\defsf} [1]{\expandafter\newcommand\csname rv#1\endcsname{{\sf #1}}}
\newcommand{\defbf} [1]{\expandafter\newcommand\csname bf#1\endcsname{{\bf #1}}}
\newcommand{\defbb} [1]{\expandafter\newcommand\csname bb#1\endcsname{{\mathbb #1}}}
\newcounter{ct}
\forLoop{1}{26}{ct}{
    \edef\letter{\Alph{ct}}
    \expandafter\defcal\letter
    \expandafter\defsf\letter
    \expandafter\defbf\letter
    \expandafter\defbb\letter
}
\forLoop{1}{26}{ct}{
    \edef\letter{\alph{ct}}
    \expandafter\defcal\letter
    \expandafter\defsf\letter
    \expandafter\defbf\letter
}

\title[Sufficient conditions for Riemannian Optimisation]{\PAPERTITLE}

\optauthor{%
\Name{Vishwak Srinivasan} \Email{vishwaks@mit.edu}\\
\Name{Ashia Wilson} \Email{ashia07@mit.edu}\\
\addr Department of Electrical Engineering and Computer Science \\ Massachusetts Institute of Technology, Cambridge, MA 02139}

\begin{document}

\maketitle

\begin{abstract}
Motivated by energy based analyses for descent methods in the Euclidean setting, we investigate a generalisation of such analyses for descent methods over Riemannian manifolds.
In doing so, we find that it is possible to derive curvature-free guarantees for such descent methods.
This also enables us to give the first known guarantees for a Riemannian cubic-regularised Newton algorithm over \(g\)-convex functions, which extends the guarantees by \citet{agarwal2021adaptive} for an adaptive Riemannian cubic-regularised Newton algorithm over general non-convex functions.
This analysis leads us to study acceleration of Riemannian gradient descent in the \(g\)-convex setting, and we improve on an existing result by \citet{alimisis2021momentum}, albeit with a curvature-dependent rate.
Finally, extending the analysis by \citet{ahn2020nesterov}, we attempt to provide some sufficient conditions for the acceleration of Riemannian descent methods in the strongly geodesically convex setting.
\end{abstract}

\section{Introduction}

In this paper, we are interested in the task of minimizing a function \(f\) defined over a Riemannian manifold \(\calM\).
This is an interesting problem, since certain \(f\) functions that are non-convex in the Euclidean sense have been shown to be convex in a Riemannian sense over a specific Riemannian manifold.
We refer to this notion as geodesic convexity or \(g\)-convexity, and formally define this later in this paper.
This therefore motivates the study of optimisation methods over Riemannian manifolds, where considerable progress has been recently made in understanding such methods, and proposing better alternatives.
\citet{zhang2016first} identified that a modified triangle equality was sufficient to obtain non-asymptotic guarantees for Riemannian gradient and subgradients methods.
This triangle inequality also underscored the study of an accelerated Riemannian gradient descent algorithm \cite{zhang2018towards}, which also used the idea of estimate sequences \cite{nesterov2018lectures} to achieve (local) acceleration within a ball around the minimizer of \(f\) for strongly \(g\)-convex functions.
While this study focused on guarantees for strongly \(g\)-convex functions, a recent paper by \citet{alimisis2021momentum} investigated acceleration of first order methods over bounded domains for a broader class of functions which include \(g\)-convex functions, and proposed an algorithm which is shown to have strictly better rate than Riemannian gradient descent, but did not achieve global acceleration.
Their analysis was motivated by a previous study on continuous-time flows to help model acceleration over Riemannian manifolds \citep{alimisis2020continuous}.
Complementary to these attempts, \citet{hamilton2021no} and \citet{criscitiello2021negative} show that global acceleration may not be achievable over negatively curved manifolds.
However, recent papers by \citet{martinez2022global} and \citet{kim2022accelerated} show that we can obtain global acceleration inside a bounded subset of the manifold, and the rates of convergence are affected by the size of this subset.
While the focus of this non-exhaustive review is first-order methods, second-order methods have also been proposed over Riemannian manifolds, and we refer to \citet[Chapter 6]{boumal2022intromanifolds} for a detailed introduction to such methods.

\subsection{Background}

In this subsection, we introduce key definitions and terminology necessary for this work.

\paragraph{Riemannian Manifolds, curves and operations}
A Riemannian manifold \((\calM, \mathfrak{g})\) is a smooth manifold \(\calM\) equipped with a Riemannian metric \(\mathfrak{g}\) that defines an inner product \(\mathfrak{g}_{x}(v, w) \equiv \langle v, w \rangle_{x}\) between two vectors \(v, w\) in the tangent space \(\calT_{x}\calM\) of \(x\) for every \(x\in \calM\).
This inner product consequently induces a norm denoted by \(\|v\|_{x} = \langle v, v\rangle_{x}\) for all \(v \in \calT_{x}\calM\).
A geodesic between two points \(x\) and \(y\) on the manifold is a locally distance minimizing curve \(\gamma_{x}^{y} : [0, 1] \to \calM\) starting at \(x\) and ending at \(y\), and the distance \(d(x, y)\) between \(x\) and \(y\) is given by the length of this geodesic.
A subset \(A\) of \(\calM\) is a geodesically unique set if for any two points in \(A\), there exists a unique geodesic connecting them.
A subset \(A\) of \(\calM\) is a geodesically convex set if for any two points in \(A\), there exists a geodesic between those points whose image lies in \(A\).
The retraction at \(x\) is a map that takes a tangent vector \(v \in \calT_{x}\calM\) and returns a point \(R_{x}(v) \in \calM\).
A special retraction is the exponential map \(\Exp_{x} : \calT_{x}\calM \to \calM\), which maps a tangent vector to the endpoint of the geodesic that starts at \(x\) with initial velocity \(v\).
The inverse of the exponential map if it exists is called the logarithmic map (\(\Log_{x} : y \mapsto v\)), which computes the initial velocity of the geodesic starting from \(x\) to reach a point \(y\).
When multiple or no geodesics exist between two points, the logarithmic map is not well-defined.
Note that when \(\calM\) is a vector space, the corresponding geodesic between two points \(x\) and \(y\) is the convex combination \(\gamma_{x}^{y}(t) = (1 - t) \cdot x + t \cdot y\).
As a result, the exponential and logarithmic maps can be viewed as generalisations to vector addition and subtraction respectively.

\paragraph{Functions over Riemannian manifolds}
The differential of a function \(f : \calM \to bbR\) at \(x\) is a real-valued function that takes in a tangent vector \(v \in \calT_{x}\calM\) and computes the infinitesimal change in function value when moving away from \(x\) in direction \(v\).
Formally,
\begin{equation*}
    Df(x)[v] = \lim_{t \to 0} \frac{f(c(t)) - f(c(0))}{t}; \quad \text{where } \dot{c}(0) = v, ~c(0) = x.
\end{equation*}
The gradient of \(f\) at \(x\) is a vector \(\grad f(x) \in \calT_{x}\calM\) that satisfies \(\langle \grad f(x), v\rangle_{x} = Df(x)[v]\) for all \(v \in \calT_{x}\calM\).
Moreover, a function is said to be differentiable over \(\calM\) if \(\grad f(x)\) exists at all points in \(\calM\) and is unique.
\(f : A \to \bbR\) is a \(g\)-convex function if \(A\) is a geodesically convex set and for any pair \(x, y \in A\), \(f(\gamma_{x}^{y}(t)) \leq (1 - t) \cdot f(x) + t \cdot f(y)\) for all \(t \in [0, 1]\).
Equivalently, when \(f\) is differentiable, for all \(x\) and \(s \in \calT_{x}\calM\) such that \(\Exp_{x}(s)\) is defined, \(f(\Exp_{x}(s)) \geq f(x) + \langle \grad f(x), s\rangle_{x}\).
\(f : A \to \bbR\) is a \(\mu\)-strongly \(g\)-convex function if \(A\) is a geodesically convex set and for any pair \(x, y\ \in A\), \(f(\gamma_{x}^{y}(t)) \leq (1 - t)\cdot f(x) + t\cdot f(y) - \frac{\mu \cdot t(1 - t)}{2}d(x, y)^{2}\).
When \(f\) is differentiable, this is equivalent to the inequality \(f(\Exp_{x}(s)) \geq f(x) + \langle \grad f(x), s\rangle_{x} + \frac{\mu}{2}\|s\|_{x}^{2}\) for all \(x, s\) such that \(\Exp_{x}(s)\) is defined.

\section{An energy based analysis of Riemannian descent methods}
\label{sec:p-descent-energy-analysis}

To study Riemannian descent methods, we first introduce an abstraction that will allow us to study a collection of such algorithms in a unified manner.
This abstraction is equivalent to \(1\)-descent algorithms of order \(p\) proposed by \citet{wilson2019accelerating} in the Euclidean case.

\begin{definition}[\(p\)-descent algorithm]
    \label{def:p-descent-manifold}
    An iterative algorithm \(\calA\) is a \(p\)-forward descent / \(p\)-backward descent algorithm in \(\calM\) w.r.t. function \(f\) if the sequence of iterates \(\{x_{k}\}_{k \geq 1}\) satisfies
    \begin{subequations}
        \setlength{\abovedisplayskip}{4pt}
        \setlength{\belowdisplayskip}{4pt}    
    \begin{align}
        f(x_{k + 1}) \leq f(x_{k}) - c\|\grad f(x_{k + 1})\|^{\nicefrac{p}{(p - 1)}}_{x_{k + 1}}, \qquad k \geq 0 & & [p\text{-forward descent}] \label{eq:p-forward-descent-manifold} \\
        f(x_{k + 1}) \leq f(x_{k}) - c\|\grad f(x_{k})\|^{\nicefrac{p}{(p - 1)}}_{x_{k}}, \qquad k \geq 0 & & [p\text{-backward descent}] \label{eq:p-backward-descent-manifold}
    \end{align}
    \end{subequations}
    where \(c\) is a constant independent of \(k\) and \(x_{0} \in \calM\) is the initialisation given to \(\calA\).
\end{definition}

For such descent algorithms, we show that it is possible to obtain rates of convergence to the minimizer \(x^{\star}\) of \(f\) where \(f\) is \(g\)-convex function, analogous to the Euclidean setting under some conditions.
Deriving such rates is possible since the descent property deals with vectors in the tangent space of a point on the manifold, and the tangent space is isomorphic to a Euclidean space where we can utilise results that hold over vector spaces.
Before stating this theorem, we introduce some assumptions as made in previous work (\cite{zhang2018towards}, \cite{alimisis2021momentum}) that we use in the proof.
\begin{assumplist1}
\item \label{assump:1} \(A\) is a geodesically unique convex subset of \(\calM\) with bounded diameter where \(\Exp\) and \(\Log\) are well-defined. 
\item \label{assump:2} \(x^{\star}\) is a minimum of \(f\), which lies inside \(A\).
\item \label{assump:3} All the iterates of the algorithm stay within \(A\).
\end{assumplist1}

When \(\calM\) is a Hadamard manifold, an example of \(A\) satisfying the assumptions is a sublevel set of \(f\) with respect to the initialisation that is bounded.
Hyperbolic spaces and many matrix manifolds are examples of practically relevant Hadamard manifolds.
Other examples of \(A\) when \(\calM\) is not a Hadamard manifold include a subset of a sphere of (Euclidean) radius \(R\) whose (geodesic) diameter is strictly less than \(\pi R\).

\begin{theorem}[Rate for \(p\)-descent algorithms over \(g\)-convex functions]
    \label{thm:p-descent-rate-g-convex}
    Let \(f : A \to \bbR\) be a \(g\)-convex function, and let \(x_{0}\) be the initialisation which belongs in \(A \subseteq \calM\).
    If \(\{x_{k}\}\) are the iterates of a \(p\)-descent algorithm (forward (Eq. \ref{eq:p-forward-descent-manifold}) or backward (Eq. \ref{eq:p-backward-descent-manifold}))
    then, assuming \ref{assump:1}, \ref{assump:2} and \ref{assump:3}, they satisfy the following guarantee
    \begin{equation*}
        \setlength{\abovedisplayskip}{4pt}
        \setlength{\belowdisplayskip}{4pt}    
        f(x_{k}) - f(x^{\star}) \leq C_{p}\frac{\diam(A)^{p}}{k^{p - 1}}, \qquad \forall ~ k \geq 0
    \end{equation*}
    where \(C_{p}\) is a constant dependent on \(p\).
\end{theorem}

In a less involved manner, it is possible to obtain rates when \(f\) is non-convex or when it is gradient dominated without any of the previously stated assumptions.

\begin{theorem}[Rate for \(p\)-descent over non-convex functions]
    \label{thm:p-descent-rate-non-convex}
    Let \(f : \calM \to \bbR\) be a non-convex function, and let \(x_{0} \in \calM\) be the initialisation.
    Then, a \(p\)-descent algorithm (forward (Eq. \ref{eq:p-forward-descent-manifold}) or backward (Eq. \ref{eq:p-backward-descent-manifold})) satisfies the following guarantee
    \begin{equation*}
        \min_{t \leq k} \|\grad f(x_{t})\|_{x_{t}} \leq \left(\frac{f(x_{0}) - f(x^{\star})}{ck}\right)^{\nicefrac{(p - 1)}{p}}
    \end{equation*}
\end{theorem}

\begin{definition}[\((\tau, p)\)-gradient dominated functions]
    \label{def:gradient-dominated_func}
    A differentiable function \(f : \calM \to \bbR\) is said to be \((\tau, p)\)-gradient dominated if \(x^{\star}\) is the global minimizer of \(f\) and for all \(x\)
    \begin{equation*}
        f(x) - f(x^{\star}) \leq \tau\|\grad f(x)\|_{x}^{\nicefrac{p}{(p - 1)}}.
    \end{equation*}
\end{definition}

\begin{theorem}[Rate for \(p\)-descent algorithms over \((\tau, p)\)-gradient dominated functions]
    \label{thm:p-descent-rate-grad-dom}
    Let \(f : \calM \to \bbR\) be a \((\tau, p)\)-gradient dominated function, and let \(x_{0} \in \calM\) be the initialisation.
    Then, a \(p\)-descent algorithm satisfies the following guarantees
    \begin{align*}
        f(x_{k}) - f(x^{\star}) \leq \left(1 + \frac{c}{\tau}\right)^{-k}(f(x_{0}) - f(x^{\star})), \qquad \forall ~ k \geq 0 & & [p\text{-forward descent}] \\
        f(x_{k}) - f(x^{\star}) \leq \left(1 - \frac{c}{\tau}\right)^{k}(f(x_{0}) - f(x^{\star})), \qquad \forall ~ k \geq 0 & & [p\text{-backward descent}]
    \end{align*}
\end{theorem}

\subsection*{Examples of Riemannian \(p\)-descent algorithms and their rates.}

Theorem \ref{thm:p-descent-rate-g-convex} allows us to immediately infer rates of convergence for popular Riemannian methods when used to optimize \(g\)-convex functions.
We give instances of such methods below, and include complete proofs of these propositions in Appendix \ref{app:rates-examples}.
\begin{enumerate}
\item The Riemannian gradient descent algorithm which generates a sequence of iterates according to the recursion \(x_{k + 1} = \Exp_{x}(-\eta \grad f(x))\) for \(k \geq 0\), is a \(2\)-backward descent method when \(f\) is \(L\)-g-smooth and \(0 < \eta < \nicefrac{2}{L}\).

\item The Riemannian proximal descent algorithm which generates a sequence of iterates according to the recursion \(x_{k + 1} = \argmin_{y \in \calM} f(y) + \frac{1}{2\eta}d^{2}(y, x_{k})\) for \(k \geq 0\), is a \(2\)-forward descent method for any \(\eta > 0\).

\item  The Riemannian cubic-regularized Newton algorithm which generates a sequence of iterates according to the recursion \(x_{k + 1} = \Exp_{x_{k}}(s_{k})\) where \(s_{k}\) satisfies 
\begin{equation*}
    \setlength{\abovedisplayskip}{3pt}
    \setlength{\belowdisplayskip}{3pt}
    m_{k}(s_{k}) \leq m_{k}(0), \enskip \|\nabla m_{k}(s_{k})\|_{x_{k}} \leq \theta\|s_{k}\|^{2}_{x_{k}}
\end{equation*}
with \(m_{k}(s) := f(x_{k}) + \langle s, \grad f(x_{k}) \rangle_{x_{k}} + \frac{1}{2}\langle s, \Hess f(x_{k})[s]\rangle_{x_{k}} + \frac{M}{3}\|s\|_{x_{k}}^{3}\) is \(3\)-forward descent algorithm when \(f\) has \(\rho\)-Lipschitz continuous Hessians and \(M > \nicefrac{\rho}{2}\).
\end{enumerate}
\noindent\fbox{%
\parbox{\textwidth}{%
\paragraph{Remark}
The rates of convergence that we attain for \(p\)-descent methods are \emph{curvature independent}, improving on the popular result for Riemannian gradient descent in \citep{zhang2016first} at the cost of some additional assumptions, and also matches the curvature independent guarantees stated in \citep{bento2017iteration} for Riemannian gradient descent and proximal descent algorithms over non-negatively curved manifolds.
}}

\noindent\fbox{%
\parbox{\textwidth}{%
\paragraph{Remark}
This theorem allows us to give the first known rates for a cubic-regularized Newton type algorithm over Riemannian manifolds for \(g\)-convex functions. The algorithm is a simpler version of the practical algorithm by \citet{agarwal2021adaptive}, which has guarantees in the non-convex setting.
}}

\section{Accelerating descent methods for \(g\)-convex and strongly \(g\)-convex functions}
\label{sec:acc-2-backward-descent-analysis}

The energy-based analysis provides an effective way to analyse accelerated versions of descent methods, as studied by \citet{wilson2021lyapunov} in the Euclidean case.
In this section, we study the acceleration of the simplest \(p\)-forward / backward descent algorithms -- which is when \(p\) is \(2\) -- through a Nesterov-style scheme.
The algorithm is composed of three updates defined below.
\begin{subequations}
\label{alg:acc-rgd-g-convex}
\begin{numcases}{}
    x_{k + 1} = \Exp_{y_{k}}(\tau_{k + 1}\Log_{y_{k}}(z_{k})) & \label{eq:update-xk} \\
    y_{k + 1} = G_{c}(x_{k + 1}) & \label{eq:update-yk} \\
    z_{k + 1} = \Exp_{x_{k + 1}}\left((\alpha_{k + 1} + \beta_{k + 1})^{-1}\left\{\beta_{k + 1}\Log_{x_{k + 1}}(z_{k}) - \grad f(x_{k + 1})\right\}\right) \label{eq:update-zk}
\end{numcases}
\end{subequations}
with \(y_{0} = z_{0} \in A \subseteq \calM\).
\(\tau_{k} \in (0, 1)\), \(\alpha_{k}, \beta_{k} > 0\) for all \(k \geq 0\).
\(G_{c}\) is a mapping which ensures that for all \(x\), \(f(G_{c}(x)) - f(x) \leq -c\|\grad f(x)\|_{x}^{2}\).
To proceed with the energy based analysis, we first define the energy function.
We use a combination of the function optimality gap and the variant of the distance to the optimum, formally defined below.
\begin{equation}
    \label{eq:energy-function}
    E_{k} = A_{k} \cdot (f(y_{k}) - f(x^{\star})) + B_{k} \cdot \left[\|\Log_{x_{k}}(z_{k}) - \Log_{x_{k}}(x^{\star})\|_{x_{k}}^{2}\right].
\end{equation}
This choice of the energy function was previously used by \citet{ahn2020nesterov} to study acceleration of Riemannian gradient descent for strongly \(g\)-convex functions.
Since the distance term of the energy is explicitly dependent on \(x_{k}\), the analysis is not straightforward, as we cannot directly compare \(\|\Log_{x_{k}}(y)\|_{x_{k}}\) and \(\|\Log_{x_{k + 1}}(y)\|_{x_{k + 1}}\) for \(y\) in general.
To aid us in proceeding with the analysis, we use the notion of a \emph{valid distortion rate} which was originally proposed by \citet{ahn2020nesterov}.
\begin{definition}[Valid Distortion Rate {\citet[Definition 3.2]{ahn2020nesterov}}]
\(\delta_{k}\) is a \emph{valid distortion rate} at iteration \(k\) if \(\|\Log_{x_{k}}(z_{k - 1}) - \Log_{x_{k}}(x^{\star})\|_{x_{k}}^{2} \leq \delta_{k}\|\Log_{x_{k - 1}}(z_{k - 1}) - \Log_{x_{k - 1}}(x^{\star})\|_{x_{k - 1}}^{2}\).
\end{definition}

\citet{ahn2020nesterov} provide computable forms of \(\delta_{k}\) based on the iterates \(x_{k}\) and \(z_{k}\) at each iteration \(k\) for Hadamard and non-Hadamard manifolds.
Assuming the existence of such valid distortion rates, we can use Equation \ref{eq:energy-function} with specific setting of parameters \(A_{k}\), \(B_{k}\), and show that the accelerated method in Equation \ref{alg:acc-rgd-g-convex} with algorithmic parameters \(\tau_{k}, \alpha_{k}, \beta_{k}\) has better rate guarantees than a standard \(2\)-forward descent / \(2\)-backward descent algorithm.

\begin{theorem}[Guarantees for \(g\)-convex functions]
\label{thm:acc-rgd-g-convex}
Let \(\{y_{k}\}\) be the sequence of \(y\) iterates generated by the algorithm described in Equation \ref{alg:acc-rgd-g-convex} when given a \(g\)-convex function \(f\), an initialisation \(y_{0} = z_{0} \in A \subseteq \calM\) and parameters
\begin{equation*}
    \begin{cases}
        & \tau_{k + 1} = \frac{2\overline{A}_{k}B_{k}}{A_{k}\delta_{k + 1}B_{k + 1} + 2B_{k}\overline{A}_{k}}; \enskip \alpha_{k + 1} = \frac{B_{k + 1} - \nicefrac{B_{k}}{\delta_{k}}}{\overline{A}_{k}}; \enskip \beta_{k + 1} = \frac{\nicefrac{B_{k}}{\delta_{k + 1}}}{\overline{A}_{k}} \\
        & \overline{A}_{k} = A_{k + 1} - A_{k}; \enskip A_{k + 1} = \frac{(k + 1)(k + 2)}{2}; \enskip B_{k + 1} = \frac{4}{c}
    \end{cases}.
\end{equation*}
Under \ref{assump:1}, \ref{assump:2} and \ref{assump:3} and assuming the existence of a valid distortion rate at every iteration \(k \geq 0\), this sequence satisfies
\begin{equation*}
    f(y_{k}) - f(x^{\star}) \leq \frac{E_{0}}{k^{2}} + \frac{\nicefrac{4}{c} \cdot \diam(A)^{2} \cdot (1 - \nicefrac{1}{\delta_{\max}})}{k}, \quad \delta_{\max} := \max_{t \leq k}\delta_{t}
\end{equation*}
for all \(k \geq 0\).
\end{theorem}

\paragraph{Remark}
Recently, \citet{alimisis2021momentum} proposed a slightly different version of the algorithm stated in Equation \ref{alg:acc-rgd-g-convex} which specifically had a geodesic search step to determine parameters \(\{\tau_{k}\}\).
As a result, their guarantee included an additional search error \(\tilde{\epsilon}\) with the rate terms.
In contrast, we find that the algorithm in Equation \ref{alg:acc-rgd-g-convex} does not require such a search step, and due to this the rate guarantee derived in Theorem \ref{thm:acc-rgd-g-convex} is free of a search error.
When \(k \leq \frac{C_{E_{0}, c}}{1 - \nicefrac{1}{\delta_{\max}}}\), the \(\nicefrac{1}{k^{2}}\) term dominates the \(\nicefrac{1}{k}\) term.
Drawing from the interpretation in \cite{alimisis2021momentum}, this can be viewed as the number of steps until which we obtain an ``accelerated'' rate.
When \(\delta_{\max} = 1\) (for e.g., when \(\calM\) is Euclidean), this upper bound is \(\infty\), and recovers the \(O(\nicefrac{1}{k^{2}})\) rate shown by \citet{nesterov1983method}.
When \(\delta_{\max} \to \infty\), then we achieve the same rate as a \(2\)-forward / backward descent algorithm.
Due to this, this algorithm achieves a strictly better rate than a standard \(2\)-forward / backward algorithm as given by this theorem.

While the preceding analysis was for (weakly) \(g\)-convex functions, we can also show that \(2\)-backward descent algorithms can be accelerated using the same algorithm in Equation \ref{alg:acc-rgd-g-convex} with a different set of algorithmic parameters.
This is direct consequence of \citet[Theorem 3.1]{ahn2020nesterov}, which was restricted to \(G_{c}(\cdot)\) being a gradient step.

\begin{proposition}[Guarantees for \(\mu\)-strongly \(g\)-convex functions]
\label{prop:acc-rgd-strongly-g-convex}
Let \(\{y_{k}\}\) be the sequence of \(y\) iterates generated by the algorithm described in Equation \ref{alg:acc-rgd-g-convex} when given a \(\mu\)-strongly \(g\)-convex function \(f : A \to \bbR\), an initialisation \(y_{0} = z_{0} \in A \subseteq \calM\) and parameters
\begin{equation*}
    \begin{cases}
    & \tau_{k + 1} = \frac{\xi_{k + 1} - 2\mu c}{1 - 2\mu c}; \enskip \alpha_{k + 1} = \mu; \enskip \beta_{k + 1} = \frac{\xi_{k + 1} - 2\mu c}{2c} \\
    & A_{k + 1} = \frac{A_{k}}{1 - \xi_{k + 1}}; \enskip B_{k + 1} = \frac{\xi_{k + 1}^{2}}{1 - \xi_{k + 1}}\cdot \frac{A_{k}}{4c}.
    \end{cases},
\end{equation*}
where \(\xi_{k + 1}\) is the solution to the equation \(\frac{\xi_{k + 1}(\xi_{k + 1} - 2\mu c)}{1- \xi_{k + 1}} = \frac{\xi_{k}^{2}}{\delta_{k + 1}}\) in \([2\mu c, 1)\) with \(A_{0}, B_{0}, \xi_{0} > 0\) and \(c < \nicefrac{1}{2\mu}\).
Under \ref{assump:1}, \ref{assump:2} and \ref{assump:3} and assuming the existence of a valid distortion rate at each iteration \(k \geq 0\), this sequence satisfies
\begin{equation*}
    f(y_{k}) - f(x^{\star}) \leq \left(\prod_{j = 1}^{k}(1 - \xi_{j})\right) \left[f(y_{0}) - f(x^{\star}) + \frac{\xi_{0}^{2}}{4c}\|\Log_{x_{0}}(z_{0}) - \Log_{x_{0}}(x^{\star})\|_{x_{0}}^{2}\right]
\end{equation*}
for all \(k \geq 0\).
\end{proposition}
\paragraph{Remark}
Note that the rate of convergence directly depends on values taken by \(\xi_{j}\), which in turn depends on the variation of the sequence of distortion rates \(\{\delta_{j}\}\).
Let \(\delta_{\max} = \max_{t \leq k}\delta_{k}\).
When \(\delta_{\max} = 1\) (for e.g., when \(\calM\) is Euclidean), \citet{ahn2020nesterov} show in their Lemma 2.1 that the sequence \(\{\xi_{k}\}_{k \geq 0}\) converges to \(\sqrt{2\mu c}\).
Thus choosing \(\xi_{0} \geq \sqrt{2\mu c}\), the sequence \(\{\xi_{k}\}_{k \geq 0}\) converges to \(\sqrt{2\mu c}\) and \(\xi_{k} \geq \sqrt{2\mu c}\) for all \(k\), giving us the rate \(\calO(\exp(-\sqrt{2\mu c} \cdot k))\).
Since \(\mu\)-strongly \(g\)-convexity corresponds to \(((2\mu)^{-1}, 2)\) gradient domination (Definition \ref{def:gradient-dominated_func}), we can compare this rate to the rate for \(2\)-backward descent algorithms over \(((2\mu)^{-1}, 2)\)-gradient dominated functions, which is \(\calO(\exp(-2\mu c\cdot k))\).
On the other extreme, when \(\delta_{\max} \to \infty\), then \(\xi_{k} \to 2\mu c\), giving us the rate \(\calO(\exp(-2\mu c \cdot k))\).
As noted for the \(g\)-convex case, these guarantees are better than one would expect from a non-accelerated version, which was noted in \citep{ahn2020nesterov} but specifically for a gradient descent step.

\subsection{Some sufficient conditions for eventual full acceleration of \(2\)-backward descent methods over \(\mu\)-strongly \(g\)-convex functions}

As noted earlier, there exists a computable sequence of valid distortion rates \(\{\delta_{k + 1}\}\) dependent on the iterates \(\{(x_{k}, z_{k})\}\) generated by the algorithm in Equation \ref{alg:acc-rgd-g-convex}.
More precisely, for Hadamard manifolds with sectional curvature lower bounded by \(-\kappa < 0\), the valid distortion rate at the \(k^{th}\) iteration is given by \(\delta_{k + 1} = T_{\kappa}(d(x_{k}, z_{k}))\) where \(T_{\kappa} : \bbR_{+} \to [1, \infty)\) is a function satisfying \(T_{\kappa}(0) = 1\).
Therefore, it would be instructive to analyse the rate at which the sequence \(\{d(x_{k}, z_{k})\}\) converges to \(0\), and translate that analysis to a rate at which the sequence \(\{\xi_{k}\}\) converges to \(\sqrt{2\mu c}\).
This is the technique adopted in \cite{ahn2020nesterov} for their analysis.
In this subsection, we extend their analysis to \(2\)-backward descent methods.
We begin by giving the following lemma, which is a generalisation of Lemma 4.2 in \cite{ahn2020nesterov}.
\begin{lemma}
\label{lem:distance-shrinking}
Let \(\calM\) be a Hadamard manifold and \(\{(x_{k}, y_{k}, z_{k})\}\) be the sequence of iterates obtained from Algorithm \ref{alg:acc-rgd-g-convex} when given a function \(f : A \to \bbR\) that is \(\mu\)-strongly \(g\)-convex and has \(L\)-Lipschitz gradients and with the parameters set to those specified in Proposition \ref{prop:acc-rgd-strongly-g-convex}.
Additionally, let the descent constant \(c\) in \(G_{c}\) (Eq. \ref{eq:update-yk}) satisfy \(c \leq \min\{\nicefrac{1}{6L}, \nicefrac{1}{2\mu}\}\).
If \(\xi_{0} \in (2\mu c, \sqrt{2\mu c}]\) and the iterates satisfy \(d(x_{k + 1}, y_{k + 1}) \leq \calC'_{L, \mu, c}\sqrt{\prod_{j = 1}^{k}(1 -\xi_{j}) \cdot D_{0}}\) for every \(k \geq 0\), then \(d(x_{k + 1}, z_{k + 1}) \leq \calC_{L, \mu, c}\sqrt{\prod_{j = 1}^{k}(1 - \xi_{j}) \cdot D_{0}}\) for every \(k \geq 0\) as well, where \(\calC_{L, \mu, c}\) and \(\calC'_{L, \mu, c}\) are constants only depending on \(L, \mu, c\) and \(D_{0} = f(y_{0}) - f(x^{\star}) + \frac{\xi_{0}^{2}}{4c}~d(z_{0}, x^{\star})^{2}\).
\end{lemma}

\paragraph{Remark} The above lemma states that with any \(2\)-backward descent method that descends sufficiently and causes the sequence of distances \(\{d(x_{k}, y_{k})\}\) to decrease at a geometric rate, then the sequence of distances \(\{d(x_{k}, z_{k})\}\) decreases at the same rate.
The original analysis by \citet{ahn2020nesterov} provides such a result when \(G_{c}\) is a gradient descent update, along with an interesting requirement that the step size be strictly greater than \(\nicefrac{1}{L}\).
Recall that for a gradient update, \(c = c(\gamma) := \gamma(1 - \nicefrac{L\gamma}{2})\) and \(\argmax_{\gamma} c(\gamma) = \nicefrac{1}{L}\).
Our lemma states that a small enough descent is sufficient for a similar geometric convergence property.

With the above lemma, we can provide a general convergence result due to a careful analysis of the evolution of the sequence \(\{\xi_{k}\}\) by \citet{ahn2020nesterov}.
\begin{proposition}[Eventual acceleration of the algorithm in Eq. \ref{alg:acc-rgd-g-convex}]
\label{prop:eventual-acc}
Let \(\{(x_{k}, y_{k})\}\) be the \((x, y)\) iterates generated by the algorithm in Equation \ref{alg:acc-rgd-g-convex} when given a function \(f : \calM \to \bbR\) that is \(\mu\)-strongly \(g\)-convex and has \(L\)-Lipschitz gradients, and with parameter settings specified in Proposition \ref{prop:acc-rgd-strongly-g-convex}.
Let the domain \(\calM\) be a Hadamard manifold with sectional curvature bounded from below by \(-\kappa < 0\).
When \(c\) satisfies \(c < \min\{\nicefrac{1}{6L}, \nicefrac{1}{2\mu}\}\), \(\xi_{0} \in (2\mu c, \sqrt{2\mu c}]\), then the sequence of iterates \(\{y_{k}\}\) generated by this algorithm satisfies \(f(y_{k}) - f(x^{\star}) \leq D_{0} \cdot \left(\prod_{j = 1}^{k}(1 - \xi_{j})\right)\) for all \(k \geq 0\).
Moreover, if \(d(x_{k+1}, y_{k+1}) \leq \calC'_{L, \mu, c}\sqrt{\prod_{j = 1}^{k}(1 - \xi_{j}) \cdot D_{0}}\) for all \(k \geq 0\), then the sequence \(\{\xi_{k}\}\) satisfies \(|\xi_{k} - \sqrt{2\mu c}| \leq \epsilon\) when \(k \geq \calC_{\kappa, L, \mu, c}\log(\nicefrac{1}{\epsilon})\) where \(\calC_{\kappa, L, \mu, c}\) is a constant depending on \(\kappa\), \(L\), \(\mu\), \(c\).
\end{proposition}
\paragraph{Remark} To achieve full acceleration, we would require \(\xi_{k} = \sqrt{2\mu c}\) for all \(k \geq 0\).
This theorem states that while we might not be able to have \(\xi_{k} = \sqrt{2\mu c}\) for all \(k \geq 0\), we can still get arbitrarily close as the algorithm proceeds, and eventually achieve acceleration.
We conjecture that this analysis will also extend to non-Hadamard manifolds under suitable assumptions (\ref{assump:1}, \ref{assump:2}, \ref{assump:3}) as discussed in \cite[Section D]{ahn2020nesterov}.

\section{Conclusion}
In this work, we presented a general analysis of Riemannian optimisation methods using an energy-based analysis framework that has gained popularity in the Euclidean setting and more recently in the Riemannian setting.
Such an analysis is also conducive to a study of accelerated first order Riemannian descent methods.
To this end, we showed that we can obtain a accelerated algorithms for first order descent methods in a straightforward manner in the \(g\)-convex and strongly \(g\)-convex setting, and present an analysis for the latter case which extends an existing analysis.
Some open questions remain: can we achieve (eventual) acceleration for a fully proximal point method, or other higher order methods such as cubic-regularized Newton even on bounded domains?

\bibliography{references.bib}

\newpage

\appendix
\raggedright
\section{More definitions}
\label{app:definitions}

With the notion of curves and geodesics, one can transport vectors in a tangent space at one point to the tangent space at another point.
This is made possible via the concept of parallel transports.
The parallel transport between \(\calT_{x}\calM\) and \(\calT_{y}\calM\) for \(x, y \in \calM\) along curve \(c\) is denoted by \(\Gamma(c)_{x}^{y} : \calT_{x}\calM \to \calT_{y}\calM\).
When \(c\) is a geodesic between \(x\) and \(y\) we omit the \(c\) in the notation and use \(\Gamma_{x}^{y}\) to simplify the notation.
A key property of parallel transports is that it is norm-preserving: for any \(v \in \calT_{x}\calM\), \(\|v\|_{x} = \|\Gamma_{x}^{y}v\|_{y}\) where \(\Gamma_{x}^{y}v \in \calT_{y}\calM\) per the definition of \(\Gamma_{x}^{y}\).
We use the parallel transport to hence define the property of \(L\)-Lipschitz gradients.
A function \(f : A \to \bbR\) is said to have \(L\)-Lipschitz gradients when it satisfies for all \(x, y \in A\),
\begin{equation*}
    \|\grad f(x) - \Gamma_{y}^{x}\grad f(y)\|_{x} \leq L\cdot d(x, y).
\end{equation*}
Such a function is also \(L\)-\(g\)-smooth i.e., for all \(x, y \in \calA\) \citep[Corollary 10.54]{boumal2022intromanifolds}
\begin{equation*}
    f(y) \leq f(x) + \langle \grad f(x), \Log_{x}(y)\rangle_{x} + \frac{L}{2}\|\Log_{x}(y)\|_{x}^{2}.
\end{equation*}

A twice differentiable function \(f : \calM \to \bbR\) is said to have \(\rho\)-Lipschitz continuous Riemannian Hessians, when for all \(x, s\) in the domain of the exponential map,
\begin{equation*}
    \left|f(\Exp_{x}(s)) - f(x) - \langle s, \grad f(x)\rangle_{x} - \frac{1}{2}\langle s, \Hess f(x)[s]\rangle_{x} \right| \leq \frac{\rho}{6}\|s\|_{x}^{3}.
\end{equation*}
Equivalently from \citep[Corollary 10.56]{boumal2022intromanifolds},
\begin{equation*}
    \left\|\left\{\Gamma_{x}^{\Exp_{x}(s)}\right\}^{-1}\grad f(\Exp_{x}(s)) - \grad f(x) - \Hess f(x)[s] \right\|_{x} \leq \frac{\rho}{2}\|s\|_{x}^{2}.
\end{equation*}

\section{Proofs for the rate theorems in Section \ref{sec:p-descent-energy-analysis}}
\label{app:p-descent-rate-proofs}

\begin{proof}[Proof of Theorem \ref{thm:p-descent-rate-g-convex}]
    We begin by noting that under the assumptions, the exponential map and its inverse exists at every \(v \in \calT_{x}\calM\) for every \(x \in A\).
    Consider an energy function
    \begin{equation*}
        E_{k} = A_{k}(f(x_{k}) - f(x^{\star})).
    \end{equation*}
    Here, \(\{A_{k}\}_{k \geq 1}\) is a sequence satisfying \(A_{k + 1} = A_{k} + a_{k}\) and \(x^{\star}\) is the minimizer of \(f\).
    The difference between \(E_{k + 1}\) and \(E_{k}\) is
    \begin{subequations}
    \begin{align}
        E_{k + 1} - E_{k} &= (A_{k} + a_{k})(f(x_{k + 1}) - f(x^{\star})) - A_{k}(f(x_{k}) - f(x^{\star})) \nonumber \\
        &= A_{k}(f(x_{k + 1}) - f(x_{k})) + a_{k}(f(x_{k + 1}) - f(x^{\star})) \label{eq:energy-forward-manifold} \\
        &= (A_{k} + a_{k})(f(x_{k + 1}) - f(x^{\star})) - (A_{k} + a_{k})(f(x_{k}) - f(x^{\star})) + a_{k}(f(x_{k}) - f(x^{\star})) \nonumber \\
        &= (A_{k} + a_{k})(f(x_{k + 1}) - f(x_{k})) + a_{k}(f(x_{k}) - f(x^{\star})) \label{eq:energy-backward-manifold}.
    \end{align}
\end{subequations}

    Since \(f\) is \(g\)-convex,
    \begin{align*}
        f(x_{k}) - f(x^{\star}) &\leq \langle \grad f(x_{k}), -\Log_{x_{k}}(x^{\star}) \rangle_{x_{k}}, \text{ and} \\
        f(x_{k + 1}) - f(x^{\star}) &\leq \langle \grad f(x_{k + 1}), -\Log_{x_{k + 1}}(x^{\star}) \rangle_{x_{k + 1}}.
    \end{align*}
    
    If \(\calA\) is a \(p\)-forward-descent algorithm w.r.t. \(f\), we can use Equation \ref{eq:p-forward-descent-manifold} and bound the difference in energies
    \begin{align}
        E_{k + 1} - E_{k} &\leq -cA_{k}\|\grad f(x_{k + 1})\|^{\nicefrac{p}{(p - 1)}}_{x_{k + 1}} + a_{k}\langle \grad f(x_{k + 1}), -\Log_{x_{k + 1}}(x^{\star})\rangle_{x_{k + 1}} \nonumber \\
        &= \frac{cA_{k}p}{p - 1} \left(\left\langle \grad f(x_{k + 1}), -\frac{a_{k}}{c A_{k}}\frac{p - 1}{p} \Log_{x_{k + 1}}(x^{\star}) \right\rangle_{x_{k + 1}}\right. \nonumber \\
        &\qquad \qquad \left.- \frac{\|\grad f(x_{k + 1})\|_{x_{k + 1}}^{\nicefrac{p}{(p -1)}}}{\nicefrac{p}{(p - 1)}}\right) \label{eq:pre-nesterov-fenchel-young-forward-manifold} 
    \end{align}
    
    If \(\calA\) is a \(p\)-backward-descent algorithm w.r.t. \(f\), we can use Equation \ref{eq:p-backward-descent-manifold} and bound the difference in energies
    \begin{align}
        E_{k + 1} - E_{k} &\leq -c(A_{k} + a_{k})\|\grad f(x_{k})\|^{\nicefrac{p}{(p - 1)}}_{x_{k}} + a_{k}\langle \grad f(x_{k}), -\Log_{x_{k}}(x^{\star}) \rangle_{x_{k}} \nonumber \\
        &= \frac{c(A_{k} + a_{k})p}{p - 1}\left(\left\langle\grad f(x_{k}), -\frac{a_{k}}{c(A_{k} + a_{k})}\frac{p - 1}{p} \Log_{x_{k}}(x^{\star})\right\rangle_{x_{k}} \right.\nonumber \\
        &\qquad \qquad \left.- \frac{\|\grad f(x_{k})\|_{x_{k}}^{\nicefrac{p}{(p - 1)}}}{\nicefrac{p}{(p - 1)}}\right) \label{eq:pre-nesterov-fenchel-young-backward-manifold} .
    \end{align}

    To bound the quantity inside the brackets in Equations \ref{eq:pre-nesterov-fenchel-young-forward-manifold} and \ref{eq:pre-nesterov-fenchel-young-backward-manifold}, we use Lemma \ref{lem:nesterov-fenchel-young}.
    Specifically, we invoke the lemma with \(q = \nicefrac{p}{p - 1}\) and
    \begin{itemize}
        \item \(\alpha = -\frac{a_{k}}{c A_{k}} \frac{p - 1}{p}\) for Equation \ref{eq:pre-nesterov-fenchel-young-forward-manifold},
        \item \(\alpha = -\frac{a_{k}}{c(A_{k} + a_{k})} \frac{p - 1}{p}\) for Equation \ref{eq:pre-nesterov-fenchel-young-backward-manifold}
    \end{itemize}
    to get
    \begin{align}
        (\ref{eq:pre-nesterov-fenchel-young-forward-manifold}) \Rightarrow E_{k + 1} - E_{k} &\leq \frac{cA_{k}p}{p(p - 1)} \cdot \left(\frac{a_{k}}{A_{k}}\right)^{p} \cdot \left(\frac{p - 1}{p}\right)^{p} \|\Log_{x_{k + 1}}(x^{\star})\|_{x_{k + 1}}^{p} \nonumber \\
        &= c'_{p} \frac{a_{k}^{p}}{A_{k}^{p - 1}} \|\Log_{x_{k + 1}}(x^{\star})\|^{p}_{x_{k + 1}} \label{eq:post-nesterov-fenchel-young-forward-manifold} \\
        (\ref{eq:pre-nesterov-fenchel-young-backward-manifold}) \Rightarrow E_{k + 1} - E_{k} &\leq \frac{c(A_{k} + a_{k})p}{p(p - 1)} \cdot \left(\frac{a_{k}}{A_{k} + a_{k}}\right)^{p} \cdot \left(\frac{p - 1}{p}\right)^{p} \|\Log_{x_{k}}(x^{\star})\|_{x_{k}}^{p} \nonumber \\
        &= c'_{p} \frac{a_{k}^{p}}{(A_{k} + a_{k})^{p - 1}} \|\Log_{x_{k}}(x^{\star})\|_{x_{k}}^{p} \label{eq:post-nesterov-fenchel-young-backward-manifold}.
    \end{align}
    where \(c'_{p} = \frac{c^{1 - p}}{p} \left(\frac{p - 1}{p}\right)^{p - 1}\).
    By definition of the exponential map, \(d(x, z) = \|\Exp_{x}^{-1}(z)\|_{x}\) for all \(x, z \in A\).
    Also, by \ref{assump:2}, \(x^{\star} \in A\).
    Therefore, \(d(z, x^{\star}) = \|\Exp_{z}^{-1}(x^{\star})\|_{z} \leq \mathrm{diam}(A)\) for any \(z \in A\).
    This further bounds of the difference in energy as
    \begin{align}
        (\ref{eq:post-nesterov-fenchel-young-forward-manifold}) \Rightarrow E_{k + 1} - E_{k} &\leq c'_{p} \frac{a_{k}^{p}}{A_{k}^{p - 1}}\diam(A)^{p} \label{eq:pre-ratio-energy-constants-forward-manifold} \\
        (\ref{eq:post-nesterov-fenchel-young-backward-manifold}) \Rightarrow E_{k + 1} - E_{k} &\leq c'_{p} \frac{a_{k}^{p}}{(A_{k} + a_{k})^{p - 1}}\diam(A)^{p} \label{eq:pre-ratio-energy-constants-backward-manifold}
    \end{align}
    Choose \(A_{k} = \frac{k(k + 1)\ldots(k + p - 1)}{p!}\).
    This gives \(a_{k} = A_{k + 1} - A_{k} = \frac{(k + 1)\ldots(k + p - 1)}{(p -1)!}\).
    Furthermore,
    \begin{align*}
        \frac{a_{k}^{p}}{A_{k}^{p - 1}} &= \underbrace{\frac{(k + 1) \ldots (k + p - 1)}{k^{p - 1}}}_{\leq p^{p - 1}} \frac{(p - 1)!^{p - 1} (p^{p -1})}{(p - 1)!^{p - 1}(p - 1)!} \leq \frac{p^{2(p - 1)}}{(p - 1)!}\\
        \frac{a_{k}^{p}}{(A_{k} + a_{k})^{p - 1}} &= \underbrace{\frac{(k + 1)\ldots(k + p - 1)}{(k + p)^{p - 1}}}_{\leq 1} \frac{(p - 1)!^{p - 1}(p^{p - 1})}{(p - 1)!^{p- 1}(p - 1)!} \leq \frac{p^{p - 1}}{(p - 1)!}
    \end{align*}
    We finally have
    \begin{align*}
        (\ref{eq:pre-ratio-energy-constants-forward-manifold}) \Rightarrow E_{k + 1} - E_{k} &\leq \underbrace{\frac{c^{1 - p}}{p} \left(\frac{p - 1}{p}\right)^{p - 1} \cdot \frac{p^{2(p - 1)}}{(p - 1)!}}_{c''_{p, \text{fwd}}} \cdot \diam(A)^{p} \\
        (\ref{eq:pre-ratio-energy-constants-backward-manifold}) \Rightarrow E_{k + 1} - E_{k} &\leq \underbrace{\frac{c^{1 - p}}{p} \left(\frac{p - 1}{p}\right)^{p - 1} \cdot \frac{p^{p - 1}}{(p - 1)!}}_{c''_{p, \text{bwd}}} \cdot \diam(A)^{p}.
    \end{align*}
    Summing both sides from \(k = 0\) to \(k = T - 1\), we get
    \begin{align*}
        E_{T} - E_{0} &\leq c''_{p, \text{fwd}} \cdot \diam(A)^{p} \cdot T \\
        \Rightarrow E_{T} &\leq c''_{p, \text{fwd}} \cdot \diam(A)^{p} \cdot T + E_{0} \Rightarrow f(x_{T}) - f(x^{\star}) \leq c''_{p, \text{fwd}} \cdot \frac{T}{A_{T}} \cdot \diam(A)^{p}, \text{ and} \\
        E_{T} - E_{0} &\leq c''_{p, \text{bwd}} \cdot \diam(A)^{p} \cdot T \\
        \Rightarrow E_{T} &\leq c''_{p, \text{bwd}} \cdot \diam(A)^{p} \cdot T + E_{0} \Rightarrow f(x_{T}) - f(x^{\star}) \leq c''_{p, \text{bwd}} \cdot \frac{T}{A_{T}} \cdot \diam(A)^{p}.
    \end{align*}
    Since \(A_{T} \geq \nicefrac{T^{p}}{p!}\), \(\nicefrac{T}{A_{T}} \leq \nicefrac{p!}{T^{p - 1}}\).
    Consequently,
    \begin{align*}
        p\text{-fwd-descent} \Rightarrow f(x_{T}) - f(x^{\star}) &\leq \frac{c^{1 - p}}{p} \left(\frac{p - 1}{p}\right)^{p - 1} \cdot \frac{p^{2(p - 1)}}{(p - 1)!} \frac{p!}{T^{p - 1}} \cdot \diam(A)^{p} \\
        &= \frac{c^{1 -p} \cdot (p^{2} - p)^{p - 1} \cdot \diam(A)^{p}}{T^{p - 1}}, \\
        p\text{-bwd-descent} \Rightarrow f(x_{T}) - f(x^{\star}) &\leq \frac{c^{1 - p}}{p} \left(\frac{p - 1}{p}\right)^{p - 1} \cdot \frac{p^{p - 1}}{(p - 1)!} \frac{p!}{T^{p - 1}} \cdot \diam(A)^{p} \\
        &= \frac{c^{1 -p} \cdot (p - 1)^{p - 1} \cdot \diam(A)^{p}}{T^{p - 1}}.
    \end{align*}
\end{proof}

\begin{proof}[Proof of Theorem \ref{thm:p-descent-rate-non-convex}]
If \(\calA\) is a \(p\)-forward descent algorithm w.r.t. \(f\), we can use Equation \ref{eq:p-forward-descent-manifold} to get
\begin{align*}
    c\|\grad f(x_{k + 1})\|_{x_{k + 1}}^{\nicefrac{p}{p - 1}} &\leq f(x_{k}) - f(x_{k + 1}) \\
    \sum_{k = 0}^{T - 1}c\|\grad f(x_{k + 1})\|_{x_{k + 1}}^{\nicefrac{p}{p - 1}} &\leq f(x_{0}) - f(x_{T}) \\
    &\leq f(x_{0}) - f(x^{\star}) \\
    \Rightarrow \min_{k \leq T} \|\grad f(x_{k})\|_{x_{k}}^{\nicefrac{p}{p - 1}} &\leq \frac{f(x_{0}) - f(x^{\star})}{cT} \\
    \Rightarrow \min_{k \leq T} \|\grad f(x_{k})\|_{x_{k}} &\leq \left(\frac{f(x_{0}) - f(x^{\star})}{cT}\right)^{\nicefrac{(p - 1)}{p}}.
\end{align*}

If \(\calA\) is a \(p\)-backward descent algorithm w.r.t. \(f\), we can use Equation \ref{eq:p-backward-descent-manifold} to get
\begin{align*}
    c\|\grad f(x_{k})\|_{x_{k}}^{\nicefrac{p}{p - 1}} &\leq f(x_{k}) - f(x_{k}) \\
    \sum_{k = 0}^{T - 1}c\|\grad f(x_{k})\|_{x_{k}}^{\nicefrac{p}{p - 1}} &\leq f(x_{0}) - f(x_{T}) \\
    &\leq f(x_{0}) - f(x^{\star}) \\
    \Rightarrow \min_{k \leq T} \|\grad f(x_{k})\|_{x_{k}}^{\nicefrac{p}{p - 1}} &\leq \frac{f(x_{0}) - f(x^{\star})}{cT} \\
    \Rightarrow \min_{k \leq T} \|\grad f(x_{k})\|_{x_{k}} &\leq \left(\frac{f(x_{0}) - f(x^{\star})}{cT}\right)^{\nicefrac{(p - 1)}{p}}.
\end{align*}
\end{proof}

\begin{proof}[Proof of Theorem \ref{thm:p-descent-rate-grad-dom}]
Consider the energy function
\begin{equation*}
    E_{k} = f(x_{k}) - f(x^{\star}).
\end{equation*}
Then, we obtain
\begin{equation*}
    E_{k + 1} - E_{k} = f(x_{k + 1}) - f(x_{k}).
\end{equation*}
If \(\calA\) is a \(p\)-forward descent algorithm w.r.t. \(f\), then using Eq. \ref{eq:p-forward-descent-manifold}
\begin{equation*}
    E_{k + 1} - E_{k} = f(x_{k + 1}) - f(x_{k}) \leq -c\|\grad f(x_{k + 1})\|_{x_{k + 1}}^{\nicefrac{p}{(p - 1)}} \leq -\frac{c}{\tau}(f(x_{k + 1}) - f(x^{\star})) = -\frac{c}{\tau}E_{k + 1}.
\end{equation*}
As a result,
\begin{equation*}
    E_{k + 1} \leq \left(1 + \frac{c}{\tau}\right)^{-1}E_{k} \Rightarrow E_{T} \leq \left(1 + \frac{c}{\tau}\right)^{-T}E_{0}.
\end{equation*}
If \(\calA\) is a \(p\)-backward descent algorithm w.r.t. \(f\), then using Eq. \ref{eq:p-backward-descent-manifold}
\begin{equation*}
    E_{k + 1} - E_{k} = f(x_{k + 1}) - f(x_{k}) \leq -c\|\grad f(x_{k})\|_{x_{k}}^{\nicefrac{p}{(p - 1)}} \leq -\frac{c}{\tau}(f(x_{k}) - f(x^{\star})) = -\frac{c}{\tau}E_{k}.
\end{equation*}
As a result,
\begin{equation*}
    E_{k + 1} \leq \left(1 - \frac{c}{\tau}\right)E_{k} \Rightarrow E_{T} \leq \left(1- \frac{c}{\tau}\right)^{T}E_{0}.
\end{equation*}
\end{proof}

\subsection{Proofs for the examples of descent methods}
\label{app:rates-examples}
\begin{proof}[Riemannian gradient descent is a \(2\)-backward descent algorithm]
    Using the assumptions, we have for \(k \geq 0\) that
    \begin{align*}
        f(x_{k + 1}) &\leq f(x_{k}) + \langle \grad f(x_{k}), \Log_{x_{k}}(x_{k + 1})\rangle_{x_{k}} + \frac{L}{2}\|\Log_{x_{k}}(x_{k + 1})\|_{x_{k}}^{2} \\
        &= f(x_{k}) - \eta \|\grad f(x_{k})\|_{x_{k}}^{2} + \frac{\eta^{2}L}{2}\|\grad f(x_{k})\|_{x_{k}}^{2}.
    \end{align*}
    When \(\eta = \nicefrac{1}{L}\), we get a simplified bound as
    \begin{equation*}
        f(x_{k + 1}) \leq f(x_{k}) - \frac{1}{2L}\|\grad f(x_{k})\|_{x_{k}}^{2}.
    \end{equation*}
\end{proof}

\begin{proof}[Riemannian proximal descent is a \(2\)-forward descent algorithm] Using the assumptions, we have for \(k \geq 0\) that
    \begin{align*}
        f(x_{k + 1}) + \frac{1}{2\eta}\|\Log_{x_{k + 1}}(x_{k})\|_{x_{k + 1}}^{2} &\leq f(x_{k}) \\
        \Rightarrow f(x_{k + 1}) &\leq f(x_{k}) - \frac{1}{2\eta}\|\Log_{x_{k + 1}}(x_{k})\|_{x_{k + 1}}^{2}.
    \end{align*}
    The proximal update also satisfies \(\Log_{x_{k + 1}}(x_{k}) = \eta\grad f(x_{k + 1})\) leading to
    \begin{equation*}
        f(x_{k + 1}) \leq f(x_{k}) - \frac{\eta}{2}\|\grad f(x_{k + 1})\|_{x_{k + 1}}^{2}.
    \end{equation*}
\end{proof}

\begin{proof}[Riemannian Cubic-regularized Newton is a \(3\)-forward descent algorithm]
    For convenience, we will denote \(\Gamma_{x}^{\Exp_{x}(s)}\) by \(P_{s}\) when operationalising the property of function with \(\rho\)-Lipschitz continuous Hessians when the choice of \(x\) is obvious
    Under our assumptions, the domain of the exponential map when restricted to \(x \in A\) is the tangent space at every point.
    At iteration \(k\), the update velocity satisfies
    \begin{equation*}
        f(x_{k}) + \langle s_{k}, \grad f(x_{k}) \rangle_{x_{k}} + \frac{1}{2}\langle s_{k}, \Hess f(x_{k})[s_{k}]\rangle_{x_{k}} + \frac{M}{3}\|s_{k}\|_{x_{k}}^{3} \leq f(x_{k}).
    \end{equation*}
    Using the fact that \(f\) has \(\rho\)-Lipschitz continuous Riemannian Hessians, we get
    \begin{align*}
        f(x_{k + 1}) &\leq f(x_{k}) + \langle s_{k}, \grad f(x_{k})\rangle_{x_{k}} + \frac{1}{2}\langle s_{k}, \Hess f(x_{k})[s_{k}]\rangle_{x_{k}} + \frac{\rho}{6}\|s_{k}\|_{x_{k}}^{3} \\
        &\leq f(x_{k}) - \left(\frac{M}{3} - \frac{\rho}{6}\right)\|s_{k}\|_{x_{k}}^{3}.
    \end{align*}
    From \citet[Theorem 3]{agarwal2021adaptive}, the gradient of \(m_{k}\) at \(s_{k}\) can be computed as 
    \begin{align*}
        \nabla m_{k}(s_{k}) &= \grad f(x_{k}) + \Hess f(x_{k})[s_{k}] + M\|s_{k}\|_{x_{k}}s_{k} \\
        &= P_{s_{k}}^{-1}\grad f(x_{k + 1}) + \grad f(x_{k}) + \Hess f(x_{k})[s_{k}] - P_{s_{k}}^{-1}\grad f(x_{k + 1}) \\
        &\qquad + M\|s_{k}\|_{x_{k}}s_{k}
    \end{align*}
    In the last step, we have added and subtracted \(P_{s_{k}}^{-1}\grad f(x_{k + 1})\).
    This leads us to,
    \begin{align*}
        \|\nabla m_{k}(s_{k})\|_{x_{k}} &= \|P_{s_{k}}^{-1}\grad f(x_{k + 1}) + \grad f(x_{k}) + \Hess f(x_{k})[s_{k}] - P_{s_{k}}^{-1}\grad f(x_{k + 1}) \\
        &\qquad + M\|s_{k}\|_{x_{k}}s_{k}\|_{x_{k}} \\
        &\geq \|P_{s_{k}}^{-1}\grad f(x_{k + 1})\|_{x_{k}} \\
        &\qquad - \|P_{s_{k}}^{-1}\grad f(x_{k + 1}) - \grad f(x_{k}) - \Hess f(x_{k})[s_{k}]\|_{x_{k}} \\
        &\qquad - M\|s_{k}\|_{x_{k}}^{2} \\
        &\geq \|\grad f(x_{k + 1})\|_{x_{k + 1}} - \frac{\rho}{2}\|s_{k}\|_{x_{k}}^{2} - M\|s_{k}\|_{x_{k}}^{2} \\
        \Rightarrow \theta\|s_{k}\|_{x_{k}}^{2} &\geq \|\grad f(x_{k + 1})\|_{x_{k + 1}} - \frac{\rho}{2}\|s_{k}\|_{x_{k}}^{2} - M\|s_{k}\|_{x_{k}}^{2} \\
    \end{align*}
    In the penultimate step, we have used the alternative characterisation of \(\rho\)-Hessian Lipschitz functions.
    Therefore,
    \begin{equation*}
    \|\grad f(x_{k + 1})\|_{x_{k + 1}} \leq \left(\theta + \frac{\rho}{2} + M\right)\|s_{k}\|_{x_{k}}^{2}.
    \end{equation*}
    Combining this with the descent statement previously, we get
    \begin{equation*}
        f(x_{k + 1}) \leq f(x_{k}) - \left(\frac{M}{3} - \frac{\rho}{6}\right)\left(\theta + \frac{\rho}{2} + M\right)^{\nicefrac{-3}{2}}\|\grad f(x_{k + 1})\|_{x_{k + 1}}^{\nicefrac{3}{2}}.
    \end{equation*}
    When \(\theta = \nicefrac{\rho}{2}, M = \rho\), we get a concise inequality
    \begin{equation*}
        f(x_{k + 1}) \leq f(x_{k}) - \frac{1}{12\sqrt{2}\sqrt{\rho}}\|\grad f(x_{k + 1})\|_{x_{k + 1}}^{\nicefrac{3}{2}}.
    \end{equation*}
\end{proof}

\section{Proof for results in Section \ref{sec:acc-2-backward-descent-analysis}}
In the proof that follow, we denote \(\|\Log_{x}(w) - \Log_{x}(v)\|_{x}\) by \(d_{x}(w, v)\) for convenience.
With this notation, the \(\delta_{k}\) is a valid distortion rate at iteration \(k\) if
\begin{equation*}
    d_{x_{k}}(z_{k - 1}, x^{\star})^{2} \leq \delta_{k}d_{x_{k - 1}}(z_{k - 1}, x^{\star})^{2}.
\end{equation*}

\subsection{Proof for convergence guarantees of Algorithm \ref{alg:acc-rgd-g-convex}}
\begin{proof}[Proof of Theorem \ref{thm:acc-rgd-g-convex}]
We analyse the difference in energy functions at iterations \(k\) and \(k + 1\).
\begin{multline*}
    E_{k + 1} - E_{k} = \underbrace{A_{k + 1} \cdot (f(y_{k + 1}) - f(x^{\star})) - A_{k} \cdot (f(y_{k}) - f(x^{\star}))}_{\Delta E^{F}_{k}} \\
    + \underbrace{B_{k + 1} \cdot d_{x_{k + 1}}(z_{k + 1}, x^{\star})^{2} - B_{k} \cdot d_{x_{k}}(z_{k}, x^{\star})^{2}}_{\Delta E_{k}^{D}}
\end{multline*}
We begin by simplifying \(\Delta E_{k}^{D}\).
First, using the fact that \(\delta_{k + 1}\) is a valid distortion rate, we get:
\begin{align*}
    \Delta E_{k}^{D} &{\leq} B_{k + 1}d_{x_{k + 1}}(z_{k + 1}, x^{\star})^{2} - \frac{B_{k}}{\delta_{k + 1}}d_{x_{k + 1}}(z_{k}, x^{\star})^{2} \\
    &= \underbrace{\left(B_{k + 1} - \frac{B_{k}}{\delta_{k + 1}}\right)}_{\overline{B}_{k}}d_{x_{k + 1}}(z_{k + 1}, x^{\star})^{2} + \frac{B_{k}}{\delta_{k + 1}}(d_{x_{k + 1}}(z_{k + 1}, x^{\star})^{2} - d_{x_{k + 1}}(z_{k}, x^{\star})^{2}) \\
\end{align*}

Next, since the tangent space \(\calT_{w}\calM\) is Euclidean for \(w \in \calM\), we have the canonical three-term lemma, which states
\begin{equation*}
    d_{w}(a, b)^{2} + d_{w}(b, c)^{2} - d_{w}(c, a)^{2} = 2\langle \Log_{w}(b) - \Log_{w}(a), \Log_{w}(b) - \Log_{w}(c)\rangle_{w}.
\end{equation*}
Using this with \(w = x_{k + 1}, a = x^{\star}, b = z_{k + 1}\) and \(c = z_{k}\) we get the bound
\begin{align*}
    \Delta E_{k}^{D} &\leq \overline{B}_{k}d_{x_{k + 1}}(z_{k + 1}, x^{\star})^{2} - \frac{B_{k}}{\delta_{k + 1}}d_{x_{k + 1}}(z_{k}, z_{k + 1})^{2} \\
    &\qquad + \frac{2B_{k}}{\delta_{k + 1}}\left(\langle \Log_{x_{k + 1}}(z_{k + 1}) - \Log_{x_{k + 1}}(x^{\star}), \Log_{x_{k + 1}}(z_{k + 1}) - \Log_{x_{k + 1}}(z_{k})\rangle_{x_{k + 1}} \right)
\end{align*}

Due to the update step \ref{eq:update-zk},
\begin{gather*}
    \frac{\alpha_{k + 1} + \beta_{k + 1}}{\beta_{k + 1}}\Log_{x_{k + 1}}(z_{k + 1}) = \Log_{x_{k + 1}}(z_{k}) - \frac{1}{\beta_{k + 1}}\grad f(x_{k + 1}) \\
    \Rightarrow \Log_{x_{k + 1}}(z_{k + 1}) - \Log_{x_{k + 1}}(z_{k}) = -\frac{\alpha_{k + 1}}{\beta_{k + 1}}\Log_{x_{k + 1}}(z_{k + 1}) - \frac{1}{\beta_{k + 1}}\grad f(x_{k + 1}).
\end{gather*}
We use this to obtain the simplification
\begin{align*}
    \Delta E_{k}^{D} &\leq \overline{B}_{k}d_{x_{k + 1}}(z_{k + 1}, x^{\star})^{2} - \frac{B_{k}}{\delta_{k + 1}}d_{x_{k + 1}}(z_{k}, z_{k + 1})^{2} \\
    &\quad + \frac{2B_{k}}{\beta_{k + 1}\delta_{k + 1}}\left\langle \Log_{x_{k + 1}}(z_{k + 1}) - \Log_{x_{k + 1}}(x^{\star}), -\alpha_{k + 1}\Log_{x_{k + 1}}(z_{k + 1}) - \grad f(x_{k + 1}) \right\rangle_{x_{k + 1}} \\
    &= \overline{B}_{k}d_{x_{k + 1}}(z_{k + 1}, x^{\star})^{2} - \frac{B_{k}}{\delta_{k + 1}}d_{x_{k + 1}}(z_{k}, z_{k + 1})^{2} \\
    &\quad + \frac{2B_{k}}{\beta_{k + 1}\delta_{k + 1}}\left(\left\langle \Log_{x_{k + 1}}(z_{k + 1}) - \Log_{x_{k + 1}}(x^{\star}), -\alpha_{k + 1}\Log_{x_{k + 1}}(z_{k + 1}) + \alpha_{k + 1}\Log_{x_{k + 1}}(x_{k + 1})\right\rangle_{x_{k + 1}}\right. \\
    &\quad \quad - \left.\left\langle \Log_{x_{k + 1}}(z_{k + 1}) - \Log_{x_{k + 1}}(x^{\star}),\grad f(x_{k + 1}) \right\rangle_{x_{k + 1}} \right) \\
    &= \overline{B}_{k}d_{x_{k + 1}}(z_{k + 1}, x^{\star})^{2} \\
    &\quad + \frac{2B_{k}\alpha_{k + 1}}{\delta_{k + 1}\beta_{k + 1}}\left\langle \Log_{x_{k + 1}}(x^{\star}) - \Log_{x_{k + 1}}(z_{k + 1}), \Log_{x_{k + 1}}(z_{k + 1}) - \Log_{x_{k + 1}}(x_{k + 1})\right\rangle_{x_{k + 1}} \\
    &\quad \quad + \frac{2B_{k}}{\delta_{k + 1}\beta_{k + 1}}\left\langle \Log_{x_{k + 1}}(x^{\star}) - \Log_{x_{k + 1}}(z_{k + 1}),\grad f(x_{k + 1}) \right\rangle_{x_{k + 1}} - \frac{B_{k}}{\delta_{k + 1}}d_{x_{k + 1}}(z_{k}, z_{k + 1})^{2}
\end{align*}
Applying the three-term lemma again with \(w = x_{k + 1}, a = x_{k + 1}, b = z_{k + 1}\) and \(c = x^{\star}\), we obtain
\begin{align*}
    \Delta E_{k}^{D} &\leq \overline{B}_{k}d_{x_{k + 1}}(z_{k + 1}, x^{\star})^{2} \\
    &\quad + \frac{B_{k}\alpha_{k + 1}}{\delta_{k + 1}\beta_{k + 1}}\left(d_{x_{k + 1}}(x_{k + 1}, x^{\star})^{2} - d_{x_{k + 1}}(z_{k + 1}, x_{k + 1})^{2} - d_{x_{k + 1}}(z_{k + 1}, x^{\star})^{2}\right) \\
    &\quad \quad + \frac{2B_{k}}{\delta_{k + 1}\beta_{k + 1}}\left\langle \Log_{x_{k + 1}}(x^{\star}) - \Log_{x_{k + 1}}(z_{k + 1}),\grad f(x_{k + 1}) \right\rangle_{x_{k + 1}} - \frac{B_{k}}{\delta_{k + 1}}d_{x_{k + 1}}(z_{k}, z_{k + 1})^{2} \\
    &= \underbrace{\left(\overline{B}_{k} - \frac{B_{k}\alpha_{k + 1}}{\delta_{k + 1}\beta_{k + 1}}\right)d_{x_{k + 1}}(z_{k + 1}, x^{\star})^{2}}_{T_{1}} - \frac{B_{k}}{\delta_{k + 1}}\left(\frac{\alpha_{k + 1}}{\beta_{k + 1}}d_{x_{k + 1}}(z_{k + 1}, x_{k + 1})^{2} + d_{x_{k + 1}}(z_{k + 1}, z_{k})^{2}\right) \\
    &\quad + \frac{2B_{k}}{\delta_{k + 1}\beta_{k + 1}}\left\langle \Log_{x_{k + 1}}(x^{\star}) - \Log_{x_{k + 1}}(z_{k + 1}),\grad f(x_{k + 1}) \right\rangle_{x_{k + 1}} + \frac{B_{k}\alpha_{k + 1}}{\delta_{k + 1}\beta_{k + 1}}d_{x_{k + 1}}(x_{k + 1}, x^{\star})^{2} \\
    &= T_{1} - \frac{B_{k}(\alpha_{k + 1} + \beta_{k + 1})}{\delta_{k + 1}\beta_{k + 1}}\left(\frac{\alpha_{k + 1}}{\alpha_{k + 1} + \beta_{k + 1}}d_{x_{k + 1}}(z_{k + 1}, x_{k + 1})^{2} + \frac{\beta_{k + 1}}{\alpha_{k + 1} + \beta_{k + 1}}d_{x_{k + 1}}(z_{k}, z_{k +1})^{2}\right) \\
    &\quad + \frac{2B_{k}}{\delta_{k + 1}\beta_{k + 1}}\left\langle \Log_{x_{k + 1}}(x^{\star}) - \Log_{x_{k + 1}}(z_{k + 1}),\grad f(x_{k + 1}) \right\rangle_{x_{k + 1}} + \frac{B_{k}\alpha_{k + 1}}{\delta_{k + 1}\beta_{k + 1}}d_{x_{k + 1}}(x_{k + 1}, x^{\star})^{2}
\end{align*}

Since the squared projected distance is effectively the squared norm of the distance between two vectors in a Euclidean space, we can use the fact that
\begin{equation*}
    \|a - \lambda b - (1 - \lambda)c\|^{2} \leq \lambda\|a - b\|^{2} + (1 - \lambda)\|a - c\|^{2}.
\end{equation*}
This is due to the convexity of the function \(f_{a}(x) = \|x - a\|^{2}\).
Using the inequality over \(\calT_{x_{k + 1}}\calM\) with \(a = \Log_{x_{k + 1}}(z_{k + 1}), b = \Log_{x_{k + 1}}(x_{k + 1}), c = \Log_{x_{k + 1}}(z_{k})\), \(\lambda = \frac{\alpha_{k + 1}}{\beta_{k + 1} + \alpha_{k + 1}}\) and \(w_{k + 1} = \lambda b + (1 - \lambda) c\) for some \(w_{k + 1} \in \calT_{x_{k + 1}}\calM\), we get
\begin{align*}
    \Delta E_{k}^{D} &\leq T_{1} - \frac{B_{k}(\alpha_{k + 1} + \beta_{k + 1})}{\delta_{k + 1}\beta_{k + 1}}\|\Log_{x_{k + 1}}(z_{k + 1}) - w_{k + 1}\|_{x_{k + 1}}^{2} + \frac{B_{k}\alpha_{k + 1}}{\delta_{k + 1}\beta_{k + 1}}d_{x_{k + 1}}(x_{k + 1}, x^{\star})^{2} \\
    &\qquad + \frac{2B_{k}}{\delta_{k + 1}\beta_{k + 1}}\left\langle \Log_{x_{k + 1}}(x^{\star}) - \Log_{x_{k + 1}}(z_{k + 1}),\grad f(x_{k + 1}) \right\rangle_{x_{k + 1}}  \\
    &= T_{1} - \frac{B_{k}(\alpha_{k + 1} + \beta_{k + 1})}{\delta_{k + 1}\beta_{k + 1}}\|\Log_{x_{k + 1}}(z_{k + 1}) - w_{k + 1}\|_{x_{k + 1}}^{2} + \frac{B_{k}\alpha_{k + 1}}{\delta_{k + 1}\beta_{k + 1}}d_{x_{k + 1}}(x_{k + 1}, x^{\star})^{2} \\
    &\qquad + \frac{2B_{k}}{\delta_{k + 1}\beta_{k + 1}}\langle \Log_{x_{k + 1}}(x^{\star}), \grad f(x_{k + 1})\rangle_{x_{k + 1}} - \frac{2B_{k}}{\delta_{k + 1}\beta_{k + 1}}\langle w_{k + 1}, \grad f(x_{k + 1})\rangle_{x_{k + 1}} \\
    &\qquad \qquad + \frac{2B_{k}}{\delta_{k + 1}\beta_{k + 1}}\langle w_{k + 1} - \Log_{x_{k + 1}}(z_{k + 1}), \grad f(x_{k + 1})\rangle_{x_{k + 1}} \\
    &\leq T_{1} - \frac{B_{k}(\alpha_{k + 1} + \beta_{k + 1})}{\delta_{k + 1}\beta_{k + 1}}\|\Log_{x_{k + 1}}(z_{k + 1}) - w_{k + 1}\|_{x_{k + 1}}^{2} + \frac{B_{k}\alpha_{k + 1}}{\delta_{k + 1}\beta_{k + 1}}d_{x_{k + 1}}(x_{k + 1}, x^{\star})^{2} \\
    &\qquad + \frac{2B_{k}}{\delta_{k + 1}\beta_{k + 1}}(f(x^{\star}) - f(x_{k + 1})) \\
    &\qquad \qquad - \frac{2B_{k}}{\delta_{k + 1}\beta_{k + 1}}\langle w_{k + 1}, \grad f(x_{k + 1})\rangle_{x_{k + 1}} + \frac{2B_{k}}{\delta_{k + 1}\beta_{k + 1}(\alpha_{k + 1} + \beta_{k + 1})}\|\grad f(x_{k + 1})\|_{x_{k + 1}}^{2}
\end{align*}
The final inequality is due to the facts that
\begin{gather*}
    w_{k + 1} = \lambda b + (1 - \lambda) c = \frac{\beta_{k + 1}}{\alpha_{k + 1} + \beta_{k + 1}}\Log_{x_{k + 1}}(z_{k})
    \\= \Log_{x_{k + 1}}(z_{k + 1}) + \frac{1}{\alpha_{k + 1} + \beta_{k + 1}}\grad f(x_{k + 1})
\end{gather*}
and that \(f\) is \(g\)-convex.

Next, note that the choice of \(\alpha_{k + 1}\) and \(\beta_{k + 1}\) satisfies \(\overline{B}_{k} = \frac{B_{k}\alpha_{k + 1}}{\delta_{k + 1}\beta_{k + 1}}\) and hence \(T_{1} = 0\).
\begin{align*}
    \Delta E_{k}^{D} &\leq T_{1} \underbrace{- \frac{B_{k}(\alpha_{k + 1} + \beta_{k + 1})}{\delta_{k + 1}\beta_{k + 1}}\|\Log_{x_{k + 1}}(z_{k + 1}) - w_{k + 1}\|_{x_{k + 1}}^{2}}_{\leq 0} + \frac{B_{k}\alpha_{k + 1}}{\delta_{k + 1}\beta_{k + 1}}d_{x_{k + 1}}(x_{k + 1}, x^{\star})^{2} \\
    &\qquad + \frac{2B_{k}}{\delta_{k + 1}\beta_{k + 1}}(f(x^{\star}) - f(x_{k + 1})) + \frac{2B_{k}}{\delta_{k + 1}\beta_{k + 1}(\alpha_{k + 1} + \beta_{k + 1})}\|\grad f(x_{k + 1})\|_{x_{k + 1}}^{2} \\
    &\qquad \qquad - \frac{2B_{k}}{\delta_{k + 1}\beta_{k + 1}}\langle w_{k + 1}, \grad f(x_{k + 1})\rangle_{x_{k + 1}} \\
    &\leq \frac{B_{k}\alpha_{k + 1}}{\delta_{k + 1}\beta_{k + 1}}d_{x_{k + 1}}(x_{k + 1}, x^{\star})^{2} + \frac{2B_{k}}{\delta_{k + 1}\beta_{k + 1}(\alpha_{k + 1} + \beta_{k + 1})}\|\grad f(x_{k + 1})\|_{x_{k + 1}}^{2} \\
    &\qquad + \frac{2B_{k}}{\delta_{k + 1}\beta_{k + 1}}(f(x^{\star}) - f(x_{k + 1})) - \frac{2B_{k}}{\delta_{k + 1}\beta_{k + 1}}\langle w_{k + 1}, \grad f(x_{k + 1})\rangle_{x_{k + 1}} \\
    &= \frac{B_{k}\alpha_{k + 1}}{\delta_{k + 1}\beta_{k + 1}}d_{x_{k + 1}}(x_{k + 1}, x^{\star})^{2} + \frac{2B_{k}}{\delta_{k + 1}\beta_{k + 1}(\alpha_{k + 1} + \beta_{k + 1})}\|\grad f(x_{k + 1})\|_{x_{k + 1}}^{2} \\
    &\qquad + \frac{2B_{k}}{\delta_{k + 1}\beta_{k + 1}}(f(x^{\star}) - f(x_{k + 1})) - \frac{2B_{k}}{\delta_{k + 1}(\alpha_{k + 1} + \beta_{k + 1})}\langle \Log_{x_{k + 1}}(z_{k}), \grad f(x_{k + 1})\rangle_{x_{k + 1}}.
\end{align*}

Due to the form of the update in Eq. \ref{eq:update-xk},
\(\Log_{x_{k + 1}}(y_{k}) = -\frac{\tau_{k + 1}}{1 - \tau_{k + 1}}\Log_{x_{k + 1}}(z_{k})\).

\begin{align*}
    \Delta E_{k}^{D} &\leq \frac{B_{k}\alpha_{k + 1}}{\delta_{k + 1}\beta_{k + 1}}d_{x_{k + 1}}(x_{k + 1}, x^{\star})^{2} + \frac{2B_{k}}{\delta_{k + 1}\beta_{k + 1}(\alpha_{k + 1} + \beta_{k + 1})}\|\grad f(x_{k + 1})\|_{x_{k + 1}}^{2} \\
    &\qquad + \frac{2B_{k}}{\delta_{k + 1}\beta_{k + 1}}(f(x^{\star}) - f(x_{k + 1})) + \frac{2B_{k}(1 - \tau_{k + 1})}{\delta_{k + 1}(\alpha_{k + 1} + \beta_{k + 1})\tau_{k + 1}}\langle \Log_{x_{k + 1}}(y_{k}), \grad f(x_{k + 1})\rangle_{x_{k + 1}} \\
    &\leq \frac{B_{k}\alpha_{k + 1}}{\delta_{k + 1}\beta_{k + 1}}(\alpha_{k + 1} - \mu)d_{x_{k + 1}}(x_{k + 1}, x^{\star})^{2} + \frac{2B_{k}}{\delta_{k + 1}\beta_{k + 1}(\alpha_{k + 1} + \beta_{k + 1})}\|\grad f(x_{k + 1})\|_{x_{k + 1}}^{2} \\
    &\qquad + \frac{2B_{k}}{\delta_{k + 1}\beta_{k + 1}}(f(x^{\star}) - f(x_{k + 1})) + \frac{2B_{k}(1 - \tau_{k + 1})}{\delta_{k + 1}(\alpha_{k + 1} + \beta_{k + 1})\tau_{k + 1}}(f(y_{k}) - f(x_{k + 1})).
\end{align*}

Finally, by definition of the constants, we have
\begin{align*}
    \Delta E_{k}^{D} &\leq \overline{B}_{k}d_{x_{k + 1}}(x_{k + 1}, x^{\star})^{2} + \frac{2\overline{A}_{k}^{2}}{B_{k + 1}}\|\grad f(x_{k + 1})\|_{x_{k + 1}}^{2} + 2\overline{A}_{k}(f(x^{\star}) - f(x_{k + 1})) \\
    &\qquad + \frac{2B_{k}(1 - \tau_{k + 1})\overline{A}_{k}}{\delta_{k + 1}B_{k + 1}\tau_{k + 1}}(f(y_{k}) - f(x_{k + 1}))
\end{align*}

Next, we look at \(\Delta E_{k}^{F}\).
\begin{align*}
    \Delta E_{k}^{F} &= A_{k + 1}(f(y_{k + 1}) - f(x_{k + 1})) + A_{k}(f(x_{k + 1}) - f(y_{k})) + \overline{A}_{k}(f(x_{k + 1}) - f(x^{\star})).
\end{align*}

As a result of these computations:
\begin{align*}
    \Delta E_{k} &\leq A_{k + 1}(f(y_{k + 1}) - f(x_{k + 1})) + \left(A_{k} - \frac{B_{k}\overline{A}_{k}(1 - \tau_{k + 1})}{\delta_{k + 1}B_{k + 1}\tau_{k + 1}}\right)(f(x_{k + 1}) - f(y_{k})) \\
    &\qquad + \frac{2\overline{A}_{k}^{2}}{B_{k + 1}}\|\grad f(x_{k + 1})\|_{x_{k + 1}}^{2} + \overline{B}_{k}d_{x_{k + 1}}(x_{k + 1}, x^{\star})^{2}.
\end{align*}
The choice of \(\tau_{k + 1}\) ensure that the coefficient of the \(f(x_{k + 1}) - f(y_{k})\) term is \(0\).
This leads to,
\begin{equation*}
    \Delta E_{k} \leq A_{k + 1}(f(y_{k + 1}) - f(x_{k + 1})) + \frac{2\overline{A}_{k}^{2}}{B_{k + 1}}\|\grad f(x_{k + 1})\|_{x_{k + 1}}^{2} + \overline{B}_{k}d_{x_{k + 1}}(x_{k + 1}, x^{\star})^{2}.
\end{equation*}
When, \(G_{c}\) is a \(2\)-backward descent method,
\(f(y_{k + 1}) - f(x_{k + 1}) \leq -c\|\grad f(x_{k + 1})\|_{x_{k + 1}}^{2}\).
This gives
\begin{equation*}
    E_{k + 1} - E_{k} \leq -\left(cA_{k + 1}- \frac{2\overline{A}_{k}^{2}}{B_{k + 1}} \right)\|\grad f(x_{k + 1})\|_{x_{k + 1}}^{2} + \overline{B}_{k}\diam(A)^{2}.
\end{equation*}
Choose \(B_{k + 1} = \frac{4}{c}\) and \(A_{k + 1} = \frac{(k + 1)(k + 2)}{2}\).
Note that for this choice \(\overline{A}_{k} = A_{k + 1} - A_{k} = (k + 1)\) and therefore, \(cA_{k + 1} > \frac{c\overline{A}_{k}^{2}}{2}\).
Due to this,
\begin{equation*}
    E_{k + 1} - E_{k} \leq \frac{4}{c}\left(1 - \frac{1}{\delta_{k + 1}}\right)\diam(A)^{2} \Rightarrow E_{T} - E_{0} \leq \frac{4T}{c} \left(1 - \frac{1}{\delta_{\max}}\right)\diam(A)^{2}.
\end{equation*}
This gives us a rate
\begin{equation*}
    f(y_{T}) - f(x^{\star}) \leq \frac{E_{0}}{A_{T}} + \frac{\frac{4}{c}\left(1 - \frac{1}{\delta_{\max}}\right)T}{A_{T}} \leq \frac{E_{0}}{T^{2}} + \frac{\frac{4}{c}\left(1 - \frac{1}{\delta_{\max}}\right)\diam(A)^{2}}{T}.
\end{equation*}
\end{proof}

\begin{proof}[Proof of Proposition \ref{prop:acc-rgd-strongly-g-convex}]
    The proof of this proposition is directly given by \citet[Theorem 3.1]{ahn2020nesterov} where we make the substitution \(\Delta_{\gamma} \to c\).
    By definition \(2\mu \Delta_{\gamma} = 2\mu\cdot \gamma(1 - \nicefrac{L\gamma}{2})\) is strictly less than \(0\) under the preconditions of their theorem, whereas due to our generality, we will have to enforce it as a property of \(G_{c}\).
    We also find that their theorem holds more generally when \(\Exp\) and \(\Log\) is well-defined at every \(x \in A\), hence the additional assumptions \ref{assump:1}, \ref{assump:2} and \ref{assump:3}.
    As noted earlier, when \(\calM\) is a Hadamard manifold, these assumptions hold.
\end{proof}

\subsection{Proofs for the sufficient conditions}

\begin{proof}[Proof of Lemma \ref{lem:distance-shrinking}]
We carefully follow the proof of Lemma 4.2 in \citet{ahn2020nesterov}.
By our assumption, \(\xi_{0} \leq \sqrt{2\mu c}\) and \(2\mu c < 1\).
For convenience, we use \(\lambda_{k + 1} = \frac{\beta_{k + 1}}{\beta_{k + 1} + \alpha_{k + 1}}\) and \(\eta_{k + 1} = \frac{1}{\beta_{k + 1} + \alpha_{k + 1}}\).

\begin{align*}
    d(x_{k + 1}, z_{k + 1}) &= \|\Log_{x_{k + 1}}(z_{k + 1})\|_{x_{k + 1}}\\
    &= \|\lambda_{k + 1}\Log_{x_{k + 1}}(z_{k}) - \eta_{k + 1}\grad f(x_{k + 1})\|_{x_{k + 1}} \\
    &\leq \lambda_{k + 1}\|\Log_{x_{k + 1}}(z_{k})\|_{x_{k + 1}} + \eta_{k + 1}\|\grad f(x_{k + 1})\|_{x_{k + 1}} \\
    &\overset{(i)}\leq \lambda_{k + 1} d(x_{k + 1}, z_{k}) + \eta_{k + 1}Ld(x_{k + 1}, x^{\star}) \\
    &\overset{(ii)}\leq \lambda_{k + 1} d(x_{k + 1}, z_{k}) + \eta_{k + 1}L d(x_{k + 1}, y_{k}) + \eta_{k + 1} L d(y_{k}, x^{\star}) \\
    &\overset{(iii)}= \lambda_{k + 1}(1 - \tau_{k + 1})d(y_{k}, z_{k}) + \eta_{k + 1}L \tau_{k + 1} d(y_{k}, z_{k}) + \eta_{k + 1} L d(y_{k}, x^{\star}) \\
    &= d(y_{k}, z_{k})(\lambda_{k + 1}(1 - \tau_{k + 1}) + \eta_{k + 1}L\tau_{k + 1}) + \eta_{k + 1} L d(y_{k}, x^{\star}).
\end{align*}
Step \((i)\) holds since \(f\) has \(L\)-Lipschitz continuous gradients.
Next, step \((ii)\) holds due to the triangle inequality over \(\calM\).
Finally, step \((iii)\) holds due to the fact that \(x_{k + 1}\) lies between \(y_{k}\) and \(z_{k}\) through Eq. \ref{eq:update-xk}.

To get a bound on \(d(x_{k + 1}, z_{k + 1})\), we need to have a bound on \(d(y_{k}, z_{k})\) and \(d(y_{k}, x^{\star})\).
We can use the energy inequality from Prop. \ref{prop:acc-rgd-strongly-g-convex} with \(\mu\)-strong g-convexity to get the following statements
\begin{gather}
    \frac{\mu}{2} \cdot d(y_{k}, x^{\star})^{2} \leq \prod_{j = 1}^{k}(1 - \xi_{j}) D_{0} \Leftrightarrow d(y_{k}, x^{\star}) \leq \sqrt{\prod_{j = 1}^{k}(1 - \xi_{j}) D_{0}}\sqrt{\frac{2}{\mu}} \label{eq:bound_d(y_k, x*)}, \\
    \mu^{2} c \cdot d_{x_{k}}(z_{k}, x^{\star})^{2} \leq \prod_{j = 1}^{k}(1 - \xi_{j}) D_{0} \Leftrightarrow d_{x_{k}}(z_{k}, x^{\star}) \leq \sqrt{\prod_{j = 1}^{k}(1 - \xi_{j}) D_{0}} \sqrt{\frac{1}{\mu^{2}c}} \label{eq:bound_proj_d(z_k, x*)}.
\end{gather}
With these we also have
\begin{align*}
    d_{x_{k}}(y_{k}, z_{k}) &\leq d_{x_{k}}(y_{k}, x^{\star}) + d_{x_{k}}(z_{k}, x^{\star}) & \because \triangle{} \text{ inequality} \\
    &\leq d(y_{k}, x^{\star}) + d_{x_{k}}(z_{k}, x^{\star}) & \because d_{a}(b, c) \leq d(b, c) \text{ for Hadamard manifolds} \\
    &\leq \sqrt{\prod_{j = 1}^{k}(1 - \xi_{j}) D_{0}}\left(\sqrt{\frac{2}{\mu}} + \sqrt{\frac{1}{\mu^{2}c}} \right) &\because \text{Eqs. } \ref{eq:bound_d(y_k, x*)}, \ref{eq:bound_proj_d(z_k, x*)}.
\end{align*}
However, this doesn't quite help us yet, since \(d_{x_{k}}(y_{k}, z_{k}) \leq d(y_{k}, z_{k})\), and we need the quantity on the RHS for the upper bound on \(d(x_{k + 1}, z_{k + 1})\).
Following the proof of \citet[Prop. C.7]{ahn2020nesterov}, we will analyse the quantity \(d_{x_{k + 1}}(y_{k + 1}, z_{k + 1})\).
\allowdisplaybreaks
\begin{align*}
    d_{x_{k + 1}}(y_{k + 1}, z_{k + 1}) &\geq -d_{x_{k + 1}}(y_{k + 1}, x_{k + 1}) + d_{x_{k + 1}}(x_{k + 1}, z_{k + 1}) \\
    &= -d_{x_{k + 1}}(y_{k + 1}, x_{k + 1}) + \|\Log_{x_{k + 1}}(z_{k + 1})\|_{x_{k + 1}} \\
    &= -d_{x_{k + 1}}(y_{k + 1}, x_{k + 1}) + \|\lambda_{k + 1}\Log_{x_{k + 1}}(z_{k}) - \eta_{k + 1}\grad f(x_{k + 1})\|_{x_{k + 1}} \\
    &\geq -d_{x_{k + 1}}(y_{k + 1}, x_{k + 1}) + \lambda_{k + 1}d(z_{k}, x_{k + 1}) - \eta_{k + 1}\|\grad f(x_{k + 1})\|_{x_{k + 1}} \\
    &\overset{(i)}\geq -d(y_{k + 1}, x_{k + 1}) + \lambda_{k + 1}(1 - \tau_{k + 1})d(y_{k}, z_{k}) - \eta_{k + 1}\|\grad f(x_{k + 1})\|_{x_{k + 1}} \\
    &\overset{(ii)}\geq -d(y_{k + 1}, x_{k + 1}) + \lambda_{k + 1}(1 - \tau_{k + 1})d(y_{k}, z_{k}) - \eta_{k + 1}L d(x_{k + 1}, x^{\star}) \\
    &\overset{(iii)}\geq -d(y_{k + 1}, x_{k + 1}) + \lambda_{k + 1}(1 - \tau_{k + 1})d(y_{k}, z_{k}) \\
    &\qquad - \eta_{k + 1}Ld(x_{k + 1}, y_{k}) - \eta_{k + 1}Ld(y_{k}, x^{\star}) \\
    &\overset{(iv)}= -d(y_{k + 1}, x_{k + 1}) + \lambda_{k + 1}(1 - \tau_{k + 1})d(y_{k}, z_{k}) \\
    &\qquad - \eta_{k + 1}L\tau_{k + 1}d(y_{k}, z_{k}) - \eta_{k + 1}Ld(y_{k}, x^{\star}).
\end{align*}
Step \((i)\) and \((iv)\) use the fact that \(x_{k + 1}\) lies between \(y_{k}\) and \(z_{k}\) by Eq. \ref{eq:update-xk}.
Step \((ii)\) uses the fact that \(f\) has \(L\)-Lipschitz continuous gradients.
Step \((iii)\) applies the triangle inequality over \(\calM\).
This gives us
\begin{align*}
    d(y_{k}, z_{k})(\lambda_{k + 1}(1 - \tau_{k + 1}) - \eta_{k + 1}L\tau_{k + 1}) &\leq d_{x_{k + 1}}(y_{k + 1}, z_{k + 1}) + d(y_{k + 1}, x_{k  +1}) + \eta_{k  +1}Ld(y_{k}, x^{\star}) \\
    &\leq \sqrt{\prod_{j = 1}^{k}(1 - \xi_{j})D_{0}}\left(\sqrt{\frac{2}{\mu}} + \sqrt{\frac{1}{\mu^{2}c}}\right) + d(y_{k + 1}, x_{k + 1}) \\
    &\qquad + \eta_{k + 1}L\sqrt{\prod_{j = 1}^{k}(1 - \xi_{j})D_{0}}\sqrt{\frac{2}{\mu}}.
\end{align*}
We make note of the fact that \(\xi_{k + 1} \leq 1\) and use the bound from Eq. \ref{eq:bound_d(y_k, x*)} and Eq. \ref{eq:bound_proj_d(z_k, x*)}.
The final piece is to bound \(d(y_{k + 1}, x_{k + 1})\) and to show that \(\lambda_{k + 1}(1 - \tau_{k + 1}) - \eta_{k + 1}L\tau_{k + 1}\) can be bounded in terms of a constant involving \(L, \mu, c\) alone.
The first part is given by the statement of the lemma, which states
\begin{equation*}
    d(y_{k + 1}, x_{k + 1}) \leq \calC'_{L, \mu, c}\sqrt{\prod_{j = 1}^{k}(1 - \xi_{j}) \cdot D_{0}}.
\end{equation*}
For the second part, we make use of global properties of the recurrence relation governing the sequence \(\{\xi_{k}\}\).
From \citet[Proposition C.9]{ahn2020nesterov}, we have that if \(\xi_{0} \leq \sqrt{a}\), then \(\xi_{k} \leq \sqrt{a}\) for all \(k \geq 0\), where
\begin{equation*}
    \frac{\xi_{k + 1}(\xi_{k + 1} - a)}{1 - \xi_{k + 1}} = \frac{\xi_{k}^{2}}{\delta}
\end{equation*}
for any \(\delta \geq 1\) and \(a \in (0, 1)\).
We use this statement with \(a = 2\mu c\), and \(\delta\) being the valid distortion rate at iteration \(k\) which is \(\geq 1\).
Therefore,
\begin{align*}
    \lambda_{k + 1}(1 - \tau_{k+1}) - \eta_{k + 1}L\tau_{k + 1} &= \left[\frac{1 - 2\mu c\xi_{k + 1}^{-1}}{1 - 2\mu c}\right](1 - \xi_{k + 1} - 2Lc) \\
    &\geq \left[\frac{1 - 2\mu c\xi_{k + 1}^{-1}}{1 - 2\mu c}\right](1 - \sqrt{2\mu c} - 2Lc).
\end{align*}
The quantity \(1 - \sqrt{2\mu c} - 2Lc\) strictly positive when \(c < \nicefrac{1}{6L}\).
Therefore, we have the bound on \(d(y_{k}, z_{k})\) as
\begin{equation*}
    d(y_{k}, z_{k}) \leq \left[\frac{1 - \sqrt{2\mu c\xi_{k + 1}^{-1}}}{1 - 2\mu c}\right]^{-1}\frac{1}{1 - \sqrt{2\mu c} - 2Lc}\calC''_{L, \mu, c}\sqrt{\prod_{j = 1}^{k}(1 - \xi_{j}) \cdot D_{0}}.
\end{equation*}
Using this to bound \(d(x_{k + 1}, z_{k + 1})\) we obtain
\begin{equation*}
    d(x_{k + 1}, z_{k + 1}) \leq \underbrace{\left(\frac{1 - 2\mu c + 2Lc}{1 - \sqrt{2\mu c} - 2Lc}\calC''_{L, \mu, c} + \frac{L\sqrt{2}}{\mu\sqrt{\mu}}\right)}_{\calC_{L, \mu, c}}\sqrt{\prod_{j = 1}^{k}(1 - \xi_{j}) \cdot D_{0}}.
\end{equation*}
\end{proof}

\begin{proof}[Proof of Proposition \ref{prop:eventual-acc}]
This proposition can be proven using the analysis of the recurrence relation as presented in \citet[Section C.7]{ahn2020nesterov}.
The key tool of the analysis is the distance shrinking lemma, which we have proven for \(2\)-backward descent methods in general when \(c\) is sufficiently small.
\end{proof}

\section{Auxiliary lemmas}
\begin{lemma}[Conjugate lemma {{\citep{nesterov2008accelerating}}}]
    \label{lem:nesterov-fenchel-young}
Let \(s, u\) be vectors in Euclidean space, and \(\alpha\) be a scalar. Then
\begin{equation*}
        \langle s, \alpha \cdot u \rangle - \frac{1}{q}\|s\|^{q} \leq \frac{q - 1}{q}|\alpha|^{\nicefrac{q}{q - 1}}\|u\|^{\nicefrac{q}{q - 1}}.
\end{equation*}
\end{lemma}
\begin{proof}
    Let \(s^{\star}\) be the maximizer of the LHS (taken with respect to \(s\)).
    By the first order optimality, we have
    \begin{equation*}
        \alpha \cdot u - \|s^{\star}\|^{q - 2}s^{\star} = 0
    \end{equation*}
    Consequently,
    \begin{equation*}
        \langle s^{\star}, \alpha \cdot u \rangle - \frac{1}{q}\|s^{\star}\|^{q} = \|s^{\star}\|^{q} - \frac{1}{q}\|s^{\star}\|^{q} = \frac{q - 1}{q}\|s^{\star}\|^{q}.
    \end{equation*}
    Also,
    \begin{equation*}
        |\alpha| \|u\| = \|s^{\star}\|^{q - 1}.
    \end{equation*}
    Hence,
    \begin{equation*}
        \langle s, \alpha \cdot u\rangle - \frac{1}{q}\|s\|^{q} \leq \langle s^{\star}, \alpha \cdot u \rangle - \frac{1}{q}\|s^{\star}\|^{q} = \frac{q - 1}{q} |\alpha|^{\nicefrac{q}{q - 1}}\|u\|^{\nicefrac{q}{q - 1}}.
    \end{equation*}
\end{proof}

\end{document}